\documentclass[11pt,reqno]{amsart}
\usepackage[capitalize]{cleveref}
\usepackage[utf8]{inputenc}
\usepackage{amssymb}
 \usepackage[T1]{fontenc}
 \usepackage{lmodern}
 \usepackage[dvips]{graphicx}
 \usepackage[all]{xy}
 \usepackage{graphicx} 
\usepackage{xspace} 

\usepackage[usenames,dvipsnames]{xcolor}
\usepackage{tikz}
\usepackage{pgfplots}

\usepackage{bm}
\usepackage{amsfonts}
\usepackage{amsmath}
\usepackage{enumerate}
\usepackage{eucal}
\theoremstyle{theorem}
\newtheorem{theorem}{Theorem}[section]
\newtheorem{proposition}[theorem]{Proposition}
\newtheorem{lemma}[theorem]{Lemma}

\theoremstyle{definition}
\newtheorem{definition}[theorem]{Definition}
\newtheorem{notation}[theorem]{Notation}
\newtheorem{remark}[theorem]{Remark}

\newcommand{\teich}{\mathcal{T}}
\newcommand{\len}{\mathrm{length}} 
\newcommand{\ml}{\mathcal{ML}}
\newcommand{\reals}{\mathbb{R}}
\newcommand{\naturals}{\mathbb{N}}
\newcommand{\dth}{d_\mathrm{Th}}
\renewcommand{\th}{\mathrm{Th}}
\newcommand{\lip}{\mathrm{Lip}}
\newcommand{\CR}{\mathrm{cr}}
\newcommand{\mcr}{\mathrm{MCR}}
\newcommand{\ie}{i.e.\ }
\newcommand{\mute}[1] {}

\newcommand{\mc}[1]{\mathcal{#1}}
\newcommand{\mf}[1]{\mathfrak{#1}}

\DeclareMathOperator{\str}{stretch}
\newcommand{\vst}{\mathbf v}

\def \R{\mathbb{R}}

\def \del{\partial}

\def \eps{\varepsilon}

\DeclareMathOperator{\Span}{Span}
\DeclareMathOperator{\cat}{cat}

\newcommand{\MF}{\ensuremath{\mathcal{MF}}\xspace}

\renewcommand{\v}{\mathbf v}

\AddToHook{env/lemma/begin}{\crefalias{theorem}{lemma}}
\AddToHook{env/corollary/begin}{\crefalias{theorem}{corollary}}
\AddToHook{env/proposition/begin}{\crefalias{theorem}{proposition}}
\AddToHook{env/remark/begin}{\crefalias{theorem}{remark}}
\title[The unit tangent spheres in Teichm\"{u}ller space]{Convex structures of the unit tangent spheres in Teichm\"{u}ller space}
\author{Assaf Bar-Natan, Ken'ichi Ohshika and Athanase Papadopoulos}
\begin{document}
\maketitle 
\sloppy

\begin{abstract}

We analyse the convex structure of the Finsler unit ball in the tangent space at each point of Teichm\"uller space of a closed surface of genus $\geq 2$ equipped with Thurston's metric.  We obtain a characterisation of faces, exposed faces and extreme points of such a unit sphere. In particular, we prove that every face has a unique naturally associated chain-recurrent geodesic lamination such that this face consists of the unit tangent vectors which are linear combinations of stretch vectors along maximal chain-recurrent geodesic laminations containing the given one. We show that a face is exposed if and only if its associated chain-recurrent geodesic lamination is the support of a measured lamination. Furthermore, we show that a point on a tangent unit sphere is an extreme point if and only if it is a stretch vector along some maximal chain-recurrent geodesic lamination. The last result gives an affirmative answer to a conjecture whose answer was known positively in the case where the surface is either the once-punctured torus or the 4-punctured sphere. Our main results also provide an alternative approach to the topological part of the infinitesimal rigidity result concerning Thurston's metric and the equivariance property of stretch vectors. 
\medskip


\medskip


\end{abstract}
\section{Introduction}
In the paper \cite{Teich}, Teichm\"{u}ller introduced the space of quasi-conformal deformations of marked Riemann surfaces  (see \cite{TeichTr} for an English translation of this paper).
He showed in particular that between any two homeomorphic marked Riemann surfaces, there is a unique best quasi-conformal homeomorphism that preserves the marking. This homeomorphism is called the Teichm\"{u}ller map. Teichm\"{u}ller  also showed that the logarithm of  the dilatation of this map induces a Finsler metric on Teichm\"{u}ller space, a metric which is now called the Teichm\"{u}ller metric. In the paper \cite{ThM}, Thurston introduced an asymmetric Finsler metric on Teichm\"{u}ller space using Lipschitz maps instead of quasi-conformal maps. This metric is now called Thurston's metric.
The last section of Thurston's paper (p.\,45-53) contains a geometric investigation of the convex structure of the Finsler unit sphere in the tangent space at every point of Teichm\"uller space with respect to this asymmetric metric.
In the present  paper, based on recent progress on the infinitesimal geometry of Thurston's metric, in particular observations made in \cite{Infinitesimal}, we analyse this convex structure more profoundly. 

To state our results concretely, we first recall some classical notions from convex geometry. References for these notions include the books \cite{Eggleston,Rock}.
Let $V$ be a finite-dimensional vector space and $\mathcal{S}$ a convex sphere in $V$, that is, the boundary of a compact convex ball $\mathcal B$ of non-empty interior in $V$. 
A convex proper subset $F\subset\mathcal B$ is said to be a \emph{face} of $\mathcal S$ if any segment in $\mathcal B$ intersecting $F$ at an interior point of this segment has both endpoints in $F$.
A \emph{support hyperplane} to $\mathcal B$ is a hyperplane $H$ separating $\mathcal B$ from its complement in $V$.
A face  of $\mathcal S$ is said to be \emph{exposed} if there exists a support hyperplane to $\mathcal B$
 such that $F= \mathcal S \cap H$.
As will be seen in \cref{convex geometry}, there are faces  of unit spheres associated with Thurston's metric which are not exposed.
Finally, we say that a point $v\in\mathcal S$ is \emph{extreme} if there is no pair $u,w$ in $\mathcal S$ such that $x$ is contained in the  segment $[u,w]$ without being an endpoint of this segment.

Let $\teich(S)$ be the Teichm\"{u}ller space  of a closed orientable surface $S$ of genus $\geq 2$. Let $x$ be a point in 
$\teich(S)$ and let
 $\mathcal S_x$ be the unit tangent sphere at $x$ with respect to the norm associated with Thurston's metric.
Our first main result, \cref{BN}, will show that any tangent vector to $\teich(S)$ at $x$ can be expressed as a linear combination of finitely many stretch vectors $\v_x(\lambda_1), \dots , \v_x(\lambda_m)$ with respect to maximal chain-recurrent geodesic laminations $\lambda_1, \dots, \lambda_m$ sharing a non-empty chain-recurrent geodesic lamination $\lambda$ as a sublamination.
Using this, we shall give a characterisation of faces, exposed faces and extreme points of the unit sphere in every tangent space to $\teich(S)$ with respect to Thurston's metric.
In \cref{face}, we prove that to every face $F$ of $\mathcal S_x$, there is a naturally associated and unique chain-recurrent geodesic lamination $\lambda$ on $S$ such that $F$ consists of the unit tangent vectors which can be expressed as  linear combinations of stretch vectors along maximal chain-recurrent geodesic laminations containing $\lambda$.
In \cref{exposed}, we show that a face $F$ is exposed if and only if its associated chain-recurrent geodesic lamination $\lambda$ is an unmeasured lamination, \ie the support of a measured lamination.
In  the last of our main theorems, \cref{extreme}, we show that a point on the unit tangent sphere $\mathcal S_x$ is extreme if and only if it is a unit stretch vector along some maximal chain-recurrent geodesic lamination.
This gives an affirmative answer to Conjecture 1.13 of \cite{Infinitesimal} which was known to be true only in the case where the surface is either the once-punctured torus or the 4-punctured sphere, see \cite{DLRT}. In the last section, by using our main results, we give an alternative approach to  the topological part of the so-called infinitesimal rigidity and the equivariance properties of stretch vectors, two results contained in the paper \cite{Infinitesimal}. The new approach does not invoke cotangent spaces neither in the statement nor in the proof, in contrast to what is done in \cite{Infinitesimal}.

Thurston's metric for the Teichm\"{u}ller space of the once-punctured torus was  studied in the papers  \cite{DLRT, HP}.
In the paper \cite{DLRT}, Dumas, Lenzhen, Rafi and Tao gave a description of the tangent spaces equipped with the norm associated with Thurston's metric in that setting, and this was one of the motivations of our present work.
A part of the present paper, in particular \S3, is adopted from the first author's PhD dissertation \cite{BN}.

The authors would like to express their gratitude to Yi Huang who served as a bridge between the first author and the other two authors  living far apart geographically.

\section{Preliminaries}
\subsection{Thurston's metric}
Let $S$ be a closed orientable surface of genus $g \geq 2$, and let $\teich(S)$ be its Teichm\"{u}ller space.
In this paper, $\teich(S)$ is regarded as the space of equivalence classes of pairs $(\Sigma, f)$  where $\Sigma$ is a hyperbolic surface and $f \colon S \to \Sigma$  a homeomorphism (the \emph{marking}) and where two pairs $(\Sigma_1, f_1)$ and $(\Sigma_2, f_2)$ are defined to be equivalent when $f_2 \circ f_1^{-1}$ is homotopic to an isometry.
We can also regard $\teich(S)$ as the space of hyperbolic metrics on $S$ modulo isotopy.
For $x \in \teich(S)$, we denote the tangent space $\teich(S)$ at $x$ by $T_x\teich(S)$. 

Throughout this paper, we study the Teichm\"{u}ller space  $\teich(S)$ equipped with Thurston's metric, introduced in \cite{ThM}. Since Thurston's metric does not satisfy the symmetry action which is usually part of the definition of a metric space, it is sometimes referred to an `asymmetric metric'.
But to save words, we use the term `metric' instead of `asymmetric metric'.
We start by recalling the definition of this metric and some of its basic properties.
For a Lipschitz map $f$ from a metric space $(X,d_X)$ to another metric space  $(Y,d_Y)$, we set $$ \lip(f)=\displaystyle\sup_{x_1\neq x_2}\frac{d_Y(f(x_1), f(x_2))}{d_X(x_1, x_2)}.$$ 
For any two points $x=(\Sigma_x, f_x)$ and $y=(\Sigma_y, f_y)$ in $\teich(S)$, the Thurston distance $\dth(x,y)$ is defined by $$\dth(x,y) = \log \inf_{f \simeq f_y \circ f_x^{-1}} \lip (f),$$
where the infimum is taken over all homeomorphisms $f$ homotopic to $f_y \circ f_x^{-1}$.
Thurston proved that this distance is also expressed as 
\begin{equation}
\label{scc}
\dth(x,y) =\log \sup_{c \in \mathcal C}\frac{\len_{\Sigma_y}(f_y(c))}{\len_{\Sigma_x}(f_x(c))},
\end{equation} 
where $\mathcal C$ denotes the set of isotopy classes of non-contractible simple closed curves on $S$, and where $\len$ denotes the length of the unique geodesic in the given homotopy class.
In the same paper, Thurston proved that the distance function is induced by a Finsler metric on Teichm\"uller space given by the family of norms  on the tangent space $T_x \teich(S)$ at each point $x$, defined by 
\begin{equation}
\Vert v \Vert_\th =\sup_{c\in \mathcal C} d\log\ell_c(v)=\sup_{c \in \mathcal C} \frac{d\ell_c(v)}{\len_{\Sigma_x}(f_x(c))},
\end{equation}
 where $\ell_c$ is the function taking $x$ to the geodesic length of $f_x(c)$ on $\Sigma_x$.
 
From now on, we shall use the shorter symbol $\len_x(c)$ to denote $\len_{\Sigma_x}(f_x(c))$, and regard a point $x$ in $\teich(S)$ as a hyperbolic structure on $S$ (rather than an equivalence class).

\subsection{Geodesic laminations}
Most notions regarding geodesic laminations on surfaces originate in Thurston's work.  We refer the reader to \cite{ThL,Ker,HP92}.
A geodesic lamination on a hyperbolic surface $\Sigma$ is a closed subset which is the union of disjoint simple geodesics. Such a geodesic, if it is not extendable, is called a leaf of the geodesic lamination.
In principle, one can talk about a geodesic lamination only on a surface equipped with a hyperbolic structure. Still, when we vary a hyperbolic structure on a surface, a geodesic lamination for the original metric is isotopic to a unique geodesic lamination on the new hyperbolic metric.
In this way, we can also regard the notion of geodesic lamination as being defined on a topological surface $S$.

Let $\lambda$ be a geodesic lamination on  $S$. 
By a \emph{component} of $\lambda$, we mean a connected component  of $\lambda$ as a subset of $S$. 
A lamination is said
to be \emph{minimal} if any nonempty sublamination of $\lambda$ is $\lambda$ itself.  
Minimal sublaminations of $\lambda$  are called {\em minimal components} of $\lambda$ although they are not necessarily components in the sense of the preceding definition.
A leaf of a geodesic
lamination is said to be \emph{isolated} if, as a subset of S, it has a neighbourhood intersecting
no leaf other than itself.  
Any geodesic lamination $\lambda$ has a unique decomposition into
finitely many minimal components and finitely many non-compact isolated leaves. 
Note that non-compact isolated leaves are not minimal components since such a leaf is not a sublamination (it is not a closed subset of S).

Chain-recurrence was introduced in Thurston's paper \cite[p.\,24-25]{ThM}.
We use, as a definition, a characterisation of chain-recurrence which is contained in the same paper.
\begin{definition}
A geodesic lamination is said to be {\em chain-recurrent} if it is a Hausdorff limit of a sequence of disjoint unions of closed geodesics.
%
%
\end{definition}
The notion of chain-recurrence does not depend on the choice of the underlying hyperbolic metric on $S$.
Any geodesic lamination $\lambda$ has a unique maximal (in the sense of inclusion) chain-recurrent sublamination, which we denote by $\lambda^\CR$.
We always have $\lambda^\CR \neq \emptyset$ since  a minimal component of $\lambda$ is always chain-recurrent.
It is easy to see by a diagonal argument that a Hausdorff limit of chain-recurrent geodesic laminations is also chain-recurrent.
%
A  chain-recurrent geodesic  lamination $\mu$
is said to be \emph{maximal} if there is no chain-recurrent geodesic lamination on $S$ properly containing it.

We shall also use the following notion. On a closed hyperbolic surface $S$, a geodesic lamination $\lambda$ is said to be \emph{complete} if every connected component of $S \setminus \lambda$ is an ideal triangle.

The following fact is mentioned by Thurston in \cite{ThM}, and we provide a proof here.
\begin{lemma}
On a closed orientable hyperbolic surface $S$, any maximal chain-recurrent geodesic lamination $\lambda$  is complete.
\end{lemma}
\begin{proof}
Since $\lambda$ is chain-recurrent, there is a sequence of disjoint unions of simple closed geodesics $\gamma_i$ converging to $\lambda$ in the Hausdorff topology.
Let us extend such a family $\gamma_i$ to a geodesic pants decomposition $\gamma_i'$. 
Adding non-compact isolated leaves  in such a way that on the two sides of each component of $\gamma'_i$ the isolated leaves spiral in opposite directions, we further extend $\gamma_i'$ to a chain-recurrent geodesic lamination $\mu_i$ whose complementary components are all ideal triangles. The space of non-empty compact subsets of $S$ equipped with the Hausdorff metric is compact. Therefore, the sequence $\mu_i$ has a convergent subsequence. Let $\lambda'$ be its Hausdorff limit.
Then $\lambda'$ is chain-recurrent, and each of its complementary components is an ideal triangle.
Since for each $i$, $\mu_i$ contains $\gamma_i$, the limit $\lambda'$ contains $\lambda$.
By the maximality of $\lambda$, we have $\lambda'=\lambda$, hence each complementary components of $\lambda$ is an ideal triangle.
%
\end{proof}

\begin{notation}  
Given a chain-recurrent geodesic lamination $\mu$, we denote by $\mathrm{MCR}(\mu)$ the set of maximal chain-recurrent geodesic laminations containing $\mu$.
In the case when $\mu$ is not chain-recurrent, we use the same notation, $\mathrm{MCR}(\mu)$, to mean $\mathrm{MCR}(\mu^\CR)$.

\end{notation}
We denote the set of all maximal chain-recurrent geodesic laminations on $S$ by $\mcr$.

Since Hausdorff limits of maximal geodesic laminations are maximal and since Hausdorff limits of chain-recurrent geodesic laminations are chain-recurrent, we have the following.

\begin{lemma}
\label{Hausdorff closed}
For any geodesic lamination $\mu$, the set $\mcr(\mu)$ is closed in the Hausdorff topology.
\end{lemma}

A {\em measured lamination} is a geodesic lamination endowed with a transverse measure of full support.
Given two measured laminations $\lambda$ and $\mu$, their intersection number $i(\lambda, \mu)$ is defined to be the total measure $\int_S d\lambda d\mu$, where $d\lambda$ and $d\mu$ denote the densities of their transverse measures.

Given a measured lamination $\lambda$, we define an associated {\em minimal supporting surface} $\Sigma(\lambda)$. This is a subsurface with non-contractible boundary components which contains $\lambda$ and which  is minimal with respect to inclusion up to isotopy among such surfaces.
We note that if $\Sigma(\lambda)$ is a minimal supporting surface of $\lambda$, then every simple closed curve $d$ with $i(\lambda , d)>0$ must intersect the boundary of $\Sigma(\lambda)$ essentially.

\subsection{Maximal ratio-maximising laminations}
We now recall the notion of maximal ratio-maximising laminations. This notion first appeared in Thurston's paper \cite[p. 37]{ThM}.
For any two points $x, y\in \teich(S)$, also regarded as hyperbolic structures on $S$, a chain-recurrent geodesic lamination $\lambda$ on $(S,x)$ is said to be \emph{ratio-maximising} from $x$ to $y$ if there is a Lipschitz homeomorphism $g\colon S \to S$ which is homotopic to the identity with Lipschitz constant equal to $\dth(x,y)$ and which stretches the leaves of $\lambda$  by precisely the factor $\exp(\dth(x,y))$.
By \cite[Theorem 8.5]{ThM}, there exists a sequence of simple closed geodesics $c_i$ on $S$ such that $\dth(x,y)$ is equal to  $\displaystyle \lim_{i\to \infty}\log\frac{\len_y(c_i)}{\len_x(c_i)}$, hence the Hausdorff limit of (a subsequence of) $(c_i)$ is ratio-maximising from $x$ to $y$.
Thurston proved that there is a unique maximal ratio-maximising chain-recurrent geodesic lamination from $x$ to $y$ (see \cite[Theorem 8.2]{ThM}), \ie a chain-recurrent geodesic lamination containing every ratio-maximising lamination from $x$ to $y$.
 Following Thurston, we denote the maximal ratio-maximising chain-recurrent geodesic lamination from $x$ to $y$ by $\mu(x,y)$.

Furthermore, Thurston showed that the following continuity property of maximal ratio-maximising laminations holds.
\begin{theorem}[Theorem 8.4 in \cite{ThM}]
\label{continuity}
Let $(x_i)$ and $(y_i)$ be two sequences in $\teich(S)$ converging to two distinct points $x$ and $y$ respectively.
Then $\mu(x,y)$ contains any Hausdorff limit of a subsequence of $(\mu(x_i, y_i))$.
\end{theorem}
This theorem has been generalised by Pan and Wolf in \cite{PW2} to the continuity of the \lq envelope', \ie the union of the geodesics connecting two endpoints, as its endpoints vary.

\subsection{Stretch maps and stretch vectors}
\label{Thurston's stretch}
In the same paper \cite{ThM}, Thurston introduced the notion of stretch map, which we shall use now.
Let $S$ be a closed orientable surface with a hyperbolic structure $m_0$ and $\lambda$ a maximal geodesic lamination on $(S, m_0)$.
Then for  $t \in [0,\infty)$, there exist continuous families of  hyperbolic structures $m_t$ on $S$ and of Lipschitz homeomorphisms $f_t: (S, m_0) \to (S, m_t)$ which are homotopic to the identity and called {\em stretch maps along $\lambda$}, such that for each $t \in [0,\infty)$, the Lipschitz constant of $f_t$ is equal to $e^t$ and is realised along the leaves of $\lambda$.
Thurston also showed in \cite{ThM} that if we regard $\str_{m_0}(\lambda,t):=(S,m_t)$ as a point of $\teich(S)$, then the ray $\{\str_{m_0}(\lambda, t)\}$, for $t$ varying in $[0,\infty)$, is a geodesic ray with respect to $\dth$, that is, $\dth(\str_{m_0}(\lambda, s), \str_{m_0}(\lambda,t))=t-s$ for all $s< t \in [0,\infty)$.
Such a ray is called a {\em stretch ray} along $\lambda$, and any of its subarcs is called a {\em stretch path}.
It is not always the case that two points in $\teich(S)$ can be connected by a stretch path, but Thurston showed that any  two points can be connected by a geodesic path which is a finite concatenation of stretch paths.

Thurston gave a description of this fact using ratio-maximising chain-recurrent geodesic laminations. For the convenience of the reader, we briefly recall Thurston's argument (see  \cite[Theorem 8.5]{ThM}).
Let $x, y$ be two distinct points in $\teich(S)$ and let $\mu(x,y)$ be the maximal ratio-maximising chain-recurrent geodesic lamination from $x$ to $y$.
Take any maximal chain-recurrent geodesic lamination $\lambda_1$ containing $\mu(x,y)$, and consider the stretch ray $\str_x(\lambda_1, t)$.
By \cref{continuity}, for any small neighbourhood $U$ of $\mu(x,y)$,  and for sufficiently small $t$, the maximal ratio-maximising lamination $\mu(\str_x(\lambda_1,t),y)$ is contained in $U$.
On the other hand,  for sufficiently small $t$, there exists a Lipschitz map from a small neighbourhood of $\lambda_1$ on $\str_x(\lambda_1,t),y)$ of Lipschitz constant $e^{d_\th(x,y)-t}$.
This implies that $\mu(\str_x(\lambda_1,t),y)$ is contained in $\lambda_1$ for sufficiently small $t$.
By \cref{continuity} again, for sufficiently small $t$, the lamination $\mu(\str_x(\lambda_1,t), y)$ cannot be larger than $\mu(x,y)$, hence it coincides with $\mu(x,y)$.
If $\str_x(\lambda_1,t)$ reaches $y$, then we are done.
Otherwise, at some point $t_1$, the lamination $\mu(\str_x(\lambda_1, t_1), y)$ becomes larger than $\mu(x,y)$ for the first time. Letting $m_1$ be $\str_x(\lambda_1, t_1)$ and $\lambda_2$  a maximal chain-recurrent geodesic lamination containing  $\mu(m_1, y)$, we consider a new stretch ray $\str_{m_1}(\lambda_2, t)$.
As before, $\mu(\str_{m_1}(\lambda_2, t),y)$ coincides with $\mu(m_1, y)$ for sufficiently small $t$, and  unless $\str_{m_1}(\lambda_1,t)$ reaches $y$, we can find $t_2$ such that $\mu(\str_{m_1}(\lambda_2, t_2),y)$ becomes larger than $\mu(m_1, y)$ for the first time.
We let $m_2$ be $\str_{m_1}(\lambda_2, t_2)$. Choosing a maximal chain-recurrent geodesic lamination $\lambda_2$ containing $\mu(m_2, y)$, we repeat the same procedure.

Thus, we have an increasing sequence of chain-recurrent geodesic laminations $\mu(x,y) \subsetneq \mu(m_1,y) \subsetneq \mu(m_2,y) \subsetneq \dots$.
Since the length of such a sequence is bounded by a number depending only on the genus of $S$, after a finite number of steps, we reach the situation where $\mu(m_j, y)$ is maximal and $\str_{m_j}(m_j, \mu(m_j,y))$ joins $m_j$ to $y$.
Thus, we have shown the following.

\begin{theorem}[Thurston]
\label{concatenation}
Any two points $x, y$ in $(\teich(S), \dth)$ can be joined by a geodesic consisting of a concatenation of stretch paths along maximal chain-recurrent geodesic laminations containing $\mu(x,y)$,  with the number of stretch paths in such a concatenation bounded by a constant depending only on the genus of $S$.
If $\mu(x,y)$ is maximal, then $x$ and $y$ can be joined by a single stretch path.
\end{theorem}

Given a stretch ray $\str_x(\lambda,t)$ along a maximal geodesic lamination $\lambda$ starting at $x\in \teich(S)$, we consider the tangent vector $\displaystyle\left.\frac{d\str_x(\lambda,t)}{dt}\right\vert_{t=0}$ at $x$.
This is called the {\em stretch vector} along $\lambda$, and we denote it by $\vst_x(\lambda)$.
It is easy to verify that $\Vert \vst_x(\lambda)\Vert_\th=1$. Not all unit tangent vectors are stretch vectors but tangent vectors are, in a very organised way, combinations of stretch vectors; see our \cref{face} below, where this organisation corresponds to the decomposition of the unit sphere into faces.

\subsection{Cataclysms}
The notion of cataclysm was introduced by Thurston in \cite[\S9]{ThM}.
Let $\lambda$ be a maximal geodesic lamination on $S$.
Given a hyperbolic metric $m$ on $S$, the lamination $\lambda$ defines a decomposition of $(S,m)$ into ideal triangles.
In each ideal triangle, there is a unique maximal (partial) foliation whose leaves are pieces of horocycles centred at the three ideal vertices of the triangle, which is symmetric with respect to the permutation of the ideal vertices and which carries a transverse measure induced by the distance function on the sides of the ideal triangle.
This partial measured foliation can be deformed into an ordinary measured foliation with a 3-prong singularity so that it fills the entire ideal triangle. See Thurston's picture in \cite[Fig. 2]{ThM}.
By pasting together these measured foliations supported in the ideal triangle, we get a measured foliation on $S$ in the usual dense, which is well defined up to isotopy, which we denote by $F_m(\lambda)$, and which we call the horocyclic foliation associated with $\lambda$ on $(S,m)$.
We denote by $\MF(S)$ the space of measured foliations on $S$ modulo Whitehead equivalence, equipped with the usual Thurston topology, that is, the topology induced by the weak topology on transverse measures.

Thurston proves in \cite[\S 9]{ThM} that for any measured foliation $F$ transverse to $\lambda$, there is a unique point $m$ in $\teich(S)$ such that $F_m(\lambda)=F$, up to Whitehead equivalence.
This gives rise to the so-called cataclysm homeomorphism $\cat \colon U \to \teich(S)$, where $U$ consists of pairs $(\lambda, F)$ in $\mcr \times \MF(S)$ such that $F$ is transverse to $\lambda$.

\subsection{The piecewise linear structure of measured foliation space and the tangent space to this space}

The space $\MF(S)$, which is known to be homeomorphic to $\reals^{6g-6} \setminus \{\mathbf 0\}$, has a piecewise smooth structure defined by train tracks.
Each chart provided by a maximal and recurrent train track is a cone in some finite-dimensional real vector space and the collection of charts form an atlas whose transition functions are piecewise linear.

In \cite[\S 6]{ThM}, Thurston shows that although the tangent space to measured lamination space does not have a natural differential structure, and therefore one cannot talk about a tangent bundle, this tangent space has `fragments' (this is Thurston's terminology) of a linear structure. 
He shows that the cataclysm map is differentiable with respect to the piecewise smooth structure on $\MF(S)$. Its differential is denoted by $d\cat \colon T(U) \to T\teich(S)$.
Each fibre on $(\lambda, F)$ is identified with $\MF(S)$.
For every $\lambda \in \mcr$ and $x \in \teich(S)$, the point $(\lambda, F_x(\lambda)) \in T(U)$ is mapped by $d\cat$ to the stretch vector $\v_x(\lambda)$.

\subsection{Convex geometry}
\label{convex geometry}
In this subsection, we shall recall some notions of convex geometry which will be used in the rest of this paper.
We shall always denote by $\mathcal B$ a convex ball with nonempty interior in some finite-dimensional vector space $V$ over $\reals$. The boundary of $\mathcal B$, denoted by $\mathcal S$, is homeomorphic to a sphere and is called a convex sphere.
We shall study the convex geometry of the unit balls and the unit spheres in the tangent spaces of Teichm\"{u}ller space equipped with the norm induced by Thurston's metric.

\begin{definition}[Face of a convex set]
A proper convex subset $F$ of the convex sphere $\mathcal B$ is said to be a {\em face} of $\mathcal S$ if  for any $v \in F$ and for any linear combination $v = tu+(1-t)w$, with $u, w \in \mathcal B$ and $t\in (0,1)$, we have $u, w \in F$.
\end{definition}

There is a special kind of face, which is called an exposed face.

\begin{definition}[Support hyperplane]
A {\em support hyperplane} for $\mathcal S$ is a codimension-one affine subspace $H$ of $V$ which intersects $\mathcal S$ but is disjoint from the interior of $\mathcal B$.
\end{definition}

\begin{definition}[Exposed face]
A subset $F$ of $\mathcal S$ is called an {\em exposed face} when $F$ is the intersection of $\mathcal S$ and a hyperplane $H$.
\end{definition}

It is easy to check that every exposed face is a face.
The converse is not true.
For instance, when $\mathcal B$ is a figure consisting of the union of a half-disc of diameter one and a square in $\reals^2$ of side one, glued along the diameter of the half-disc and a side of the square,
then an extremity of the resulting figure which is point at which the semi-circle intersects the square is a face, but it is not exposed. 

Another notion which is important in the rest of this paper is the following.
\begin{definition}
A point $v$ on $\mathcal S$ is said to be an {\em extreme point} if whenever  $v=tu +(1-t)w$ for two points $u, v \in \mathcal S$ and $t \in [0,1]$,  we have either $t=0$ or $t=1$.
\end{definition}

\section{Harmonic stretch maps and linear combinations of stretch vectors}
\subsection{Results on harmonic stretch maps and harmonic stretch vectors}
As we remarked in the previous section, it is not true that two arbitrary points in $\teich(S)$ can be joined by a stretch path.
Pan and Wolf in \cite{PW} considered a generalisation of stretch maps/paths, which they call harmonic stretch maps/paths. They showed that any two points in Teichm\"uller space can be joined by a harmonic stretch path.
We state their results without formally defining harmonic stretch maps.

\begin{theorem}[Pan--Wolf \cite{PW}]
\label{PW}
Let $\lambda$ be a  geodesic lamination on $(S,m_0)$.
Then there exists a family of Lipschitz homeomorphisms $h_t \colon (S, m_0) \to (S, m_t)$, $t\geq 0$,  called harmonic stretch maps, such that for any $t >0$, the map  $h_t$ stretches $\lambda$ by the factor $e^t$ and the restriction map $h_t | \big( (S, m_0) \setminus \lambda\big)$, whose image is equal to $ \big((S, m_t) \setminus \lambda\big)$,  has Lipschitz constant strictly less than $e^t$.
Furthermore, the family $\xi(t):=(S, m_t)$ is a geodesic ray with respect to $\dth$.
\end{theorem}

We call a ray such as $\xi(t)$  above a {\em harmonic stretch ray} along $\lambda$ and its subarcs harmonic stretch paths along $\lambda$.
A harmonic stretch ray is differentiable and defines a unit tangent vector at its starting point. This vector is called the {\em harmonic stretch vector} along $\lambda$.
The following is an essential property of harmonic stretch paths.

\begin{theorem}[Theorem 1.6 in Pan--Wolf \cite{PW}]
\label{harmonic geodesic}
For any distinct two points $x, y \in \teich(S)$, there exists a unique harmonic stretch path along the maximal ratio-maximising chain recurrent lamination $\mu(x,y)$, connecting $x$ with $y$.
\end{theorem}

A similar result at the infinitesimal level was also proved.

\begin{proposition}[Proposition 13.12 in Pan--Wolf \cite{PW}]
\label{cor:PW}
Any vector $v \in  T_x \teich(S)$ with $\Vert v\Vert_{Th}=1$ can be expressed as a harmonic stretch vector along some geodesic lamination.
\end{proposition}

This enables us to introduce an infinitesimal version of maximal ratio-maximising lamination.
\begin{definition}[infinitesimally most stretched chain recurrent geodesic lamination]
Let $v$ be a unit tangent vector in $(T_x\teich(S), \Vert \cdot \Vert_\th)$.
By \cref{cor:PW}, $v$ is expressed as a harmonic stretch vector along some  geodesic lamination $\lambda$.
We say that a chain recurrent geodesic lamination contained in $\lambda$ is {\em infinitesimally most stretched}  by $v$.
In particular, $\lambda$ is called a {largest infinitesimally most stretched geodesic lamination}, and $\lambda^\CR$ is called a {largest infinitesimally most stretched chain-recurrent geodesic lamination}.
\end{definition}

We shall show that $\lambda$ depends only on $v$ and is independent of the choice of a harmonic stretch vector.


\begin{lemma}
\label{unique stretch}
Suppose that two geodesic laminations $\lambda$ and $\mu$ are infinitesimally most  stretched laminations by the same unit vector $v$.
Then we have $\lambda=\mu$.
In particular, two geodesic laminations infinitesimally most stretched by the same vector cannot intersect transversely.
\end{lemma}
\begin{proof}
It is evident that the second statement follows from the first statement.
We prove the first statement.
Suppose that both $\lambda$ and $\mu$ are largest infinitesimally most stretched laminations by $v$, and let $r_\lambda(t)$ and $r_\mu(t)$ be two harmonic stretch rays which maximally stretch $\lambda$ and $\mu$ respectively and which are tangent to $v$ at $t=0$.
Denote $r_\lambda(t)$  by $m_t$.
Then, there exists a harmonic stretch map $f_t \colon (S,x) \to (S, m_t)$ such that $f_t|\lambda$ stretches $\lambda$ by the factor $e^t$.

Suppose that there exists a point on one of $\lambda,\mu$ which is not contained in the other.
Then we can assume, by exchanging the names if necessary, that $\mu$ has a point which is not contained in $\lambda$.
As was shown in the proof of  \cite[Lemma 7.3]{PW}, we can take a neighbourhood $U$ of a point $p$ in $\mu \setminus \lambda$, which we can assume to be contained in $S \setminus \lambda$, so that there exists a positive constant $\varepsilon$ such that the Lipschitz constant of $f_t|U$ is bounded by $e^{t-\epsilon e^t}$, whose derivative is less than $1$ at $t=0$.
On the other hand, the Lipschitz constant on the entire surface $S$ is bounded by $e^t$.
Therefore, $\mu$ is stretched at a speed strictly less than $1$.
This contradicts the facts that $r_\mu(t)$ stretches a component of $\mu$ in $S \setminus \lambda$ at unit speed, and that $r_\lambda(t)$ is tangent to $r_\mu(t)$ at $t=0$.
Thus we have $\lambda=\mu$.
\end{proof}

From this lemma, we see that for every unit vector $v \in T_x \teich(S)$, there is a unique largest infinitesimally most stretched chain-recurrent geodesic lamination.

The last property of \cref{unique stretch} characterises stretch vectors.

\begin{lemma}
\label{maximal then stretch}
If the infinitesimally most stretched chain-recurrent geodesic lamination of a unit vector $v$ is maximal, then $v$ is a stretch vector.
\end{lemma}
\begin{proof}
If a harmonic stretch ray stretches a maximal chain-recurrent geodesic lamination, then it coincides with Thurston's stretch ray, as was shown in \cite[Theorem 1.1]{PW}.
Our claim follows immediately.
\end{proof}

It is not obvious from the definition that a linear combination of two vectors which infinitesimally stretch a geodesic lamination $\lambda$ also infinitesimally stretch $\lambda$.
Still, we can prove as follows that this is the case by the same argument as the proof of \cref{unique stretch}.

\begin{lemma}
\label{linear combination}
Let $u$ and $v$ be unit tangent vectors in $T_x \teich(S)$ which infinitesimally most stretch a geodesic lamination $\lambda$.
Then for any $\alpha \in [0,1]$, the unit tangent vector $\alpha u+(1-\alpha)v$ also infinitesimally most stretches $\lambda$.
\end{lemma}
\begin{proof}
Let $r(t)$ be a harmonic stretch ray which is tangent to $w:=\alpha u+ (1-\alpha)v$ at $t=0$.
Let $\mu$ be the geodesic lamination which is stretched by $r(t)$.
Since $\mu$ is the largest infinitesimally most stretched geodesic lamination of $w$, by the same argument used in the proof of \cref{unique stretch}, we see that $\lambda$ must be contained in $\mu$.
This means, by definition, that $\lambda$ is infinitesimally most stretched by $w$.
\end{proof}

The main result of this section is the following theorem.
\begin{theorem}
\label{BN}
A unit vector $v\in T_x\teich(S)$ whose largest infinitesimally most stretched geodesic lamination is  $\lambda$ can be expressed as a linear combination $\alpha_1 \vst_x(\lambda_1)+ \dots + \alpha_k \vst_x(\lambda_k)$, where $\lambda_1, \dots , \lambda_k$ lie in $\mcr(\lambda^\CR)$, \ie they are maximal chain-recurrent geodesic laminations containing  $\lambda^\CR$.
\end{theorem}


\subsection{A technical lemma}
To prove \cref{BN}, we need to show a technical lemma
about oriented 1-foliations and their tangent vectors.
By an \textit{oriented 1-foliation}, we mean a one-dimensional $C^1$-foliation (by which we mean that every leaf is $C^1$) on a smooth manifold $M$ such that every leaf is oriented, and such that the orientations on the leaves vary continuously along any path transverse to the foliation.
For each oriented 1-foliation, we can construct a vector field of unit tangent vectors to the foliation so that the positive direction of the vector coincides with the orientation of the foliation at each point. 
We  refer to the flow along the
vector field associated with the foliation as the `flow along the
1-foliation'.

Let $\mc{F}_i\, (i \in I)$ be a family of oriented  1-foliations defined on an open domain  $ U \subset \reals^k$ containing the origin $\mathbf 0$,  with smooth (possibly empty) boundary. 
(Note that we do not assume that $I$ is countable.)
For $U^{\eps} = \{\mathbf y\in U: d(\mathbf y,\del U) >\eps\}$ with  $\eps>0$, 
we let $f_i:[0,\eps)\times U^{\eps} \to \R^k$ be a map obtained by setting $f_i(t,\mathbf x)$ to be the flow  for time $t$ along
$\mc{F}_i$ starting from $\mathbf x$. 
Taking a sufficiently small $\eps$, 
we can always ensure that $\mathbf 0 \in U^{\eps}$ and that 
all  the flows $f_i$ are defined at $\mathbf x=\mathbf 0$.
For any set of tangent vectors $\{v_i(t)\}$ at a point $(t,\mathbf 0) \in [0,\infty) \times U$, we define $\Span(\{v_i(t)\}_{i \in I})$ to be the set \begin{equation*}
\begin{split}
\{\alpha_{i_1}v_{i_1}(t)+ \dots + \alpha_{i_n} v_{i_n}(t) \mid  \ &\alpha_{i_1}, \dots , \alpha_{i_n} > 0; \alpha_{i_1}+ \dots +\alpha_{i_n}=1;  \\& n \in \naturals; i_1, \dots , i_n \in I\}.
\end{split}
\end{equation*}
We note that even  if we relax the condition $\alpha_{i_1}, \dots ,\alpha_{i_n}>0$ to $\alpha_{i_1} , \dots , \alpha_{i_n} \geq 0$, we have the same space, since at least one of these scalars must be positive.

\fussy
The following basic lemma will be essentially used in the proof of \cref{BN}.
In this lemma, we use the term \lq piecewise' $C^1$ with respect to $\mathbf x \in U^\eps$ near  $\mathbf 0$.
The meaning is as follows.
Consider the tangent space $T_{\mathbf x} U^\eps$ for $\mathbf x$ near $\mathbf 0$.
There is a decomposition of $T_{\mathbf x} U^\eps$ into finitely many closed conical regions meeting along their boundaries. The decomposition varies continuously with respect to $\mathbf x$.
A map is said to be piecewise $C^1$ if it has  directional derivatives which  are continuous near $\mathbf 0$ and partially differentiable with respect to $t$ if the directions  are restricted to any one of the conical regions.

\begin{lemma} 
\label{flows}
In the above setting of a sequence of flows $\{f_i\}$ and their tangent vectors $\{v_i(t)\}$,
suppose that $f_i(t,\mathbf x)$ is $C^1$ with respect to $t$, and piecewise $C^1$ with respect to $\mathbf x\in U^\epsilon$ near $\mathbf 0$.
Suppose moreover that both families $\{f_i(t, \mathbf x)\}_{i \in I}$ and $\displaystyle\{\frac{\partial f_i}{\partial t}(t, \mathbf x)\}_{i \in I}$  are compact with respect to the uniform topology on any compact subset of $[0,\eps) \times U^\eps$.
  Let $\alpha:[0,T) \to U$ be a path with $\alpha(0) = \mathbf 0$ which is differentiable at $0$ and such that for every $t< T$, there exist 
  $t_1,\hdots,t_n\ge 0$ 
  and $i_1,\hdots,i_n \in \mathbb N$ satisfying   $t = \sum_i t_i$ and $\alpha(t)
  = f_{i_1}(t_1,f_{i_2}(t_2,(\hdots,(f_{i_n}(t_n,\mathbf 0)))))$, where $n$ is an    integer which is independent of $t$ whereas $i_1, \dots , i_n$ may depend on $t$. 
  Then  $\alpha'(0)$ lies in $\Span(\{v_i(0)\}).$ 
\end{lemma}

\begin{figure}
  \begin{center}
    \includegraphics[width=0.6\textwidth]{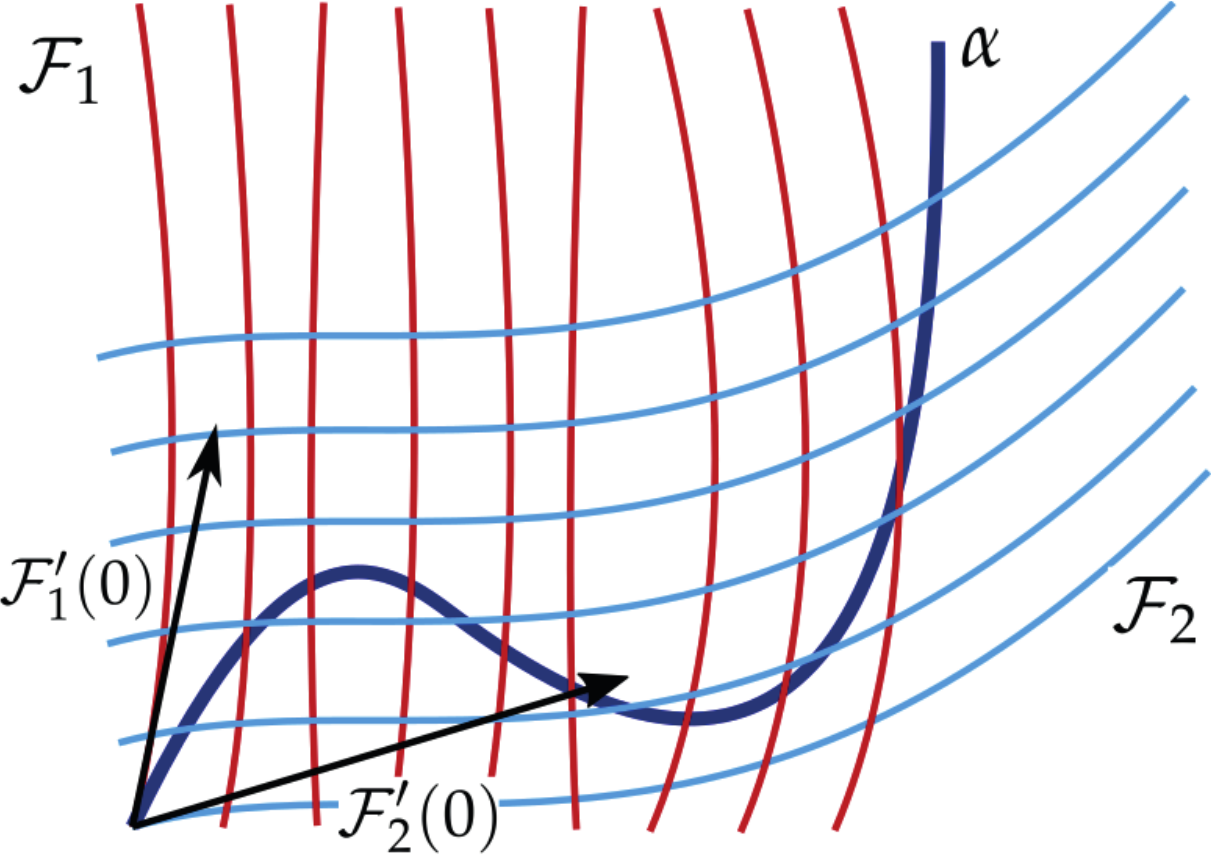}
  \end{center}
  \caption{The picture of the set-up in Lemma \ref{flows}}
\end{figure}

\begin{proof}
For small positive $t$, we can take a first-order expansion of
  $\alpha(t) = f_{i_1}(t_1,f_{i_2}(t_2,(\hdots,(f_{i_n}(t_n,\mathbf 0)))))$ with respect to $t_1$  to get
\begin{equation}
  \begin{split}
  \label{expand1}
    \alpha(t) = &f_{i_1}(0,f_{i_2}(t_2,(\hdots,(f_{i_n}(t_n,\mathbf 0))))) \\
    &+ t_1 \frac{\del f_{i_1}}{\del t}(0, f_{i_2}(t_2,(\hdots,(f_{i_n}(t_n,\mathbf 0)))))+o(t)
   \\ 
    = &f_{i_2}(t_2,(\hdots,(f_{i_n}(t_n,\mathbf 0))))\\
   & + t_1 \frac{\del f_{i_1}}{\del t}(0, f_{i_2}(t_2,(\hdots,(f_{i_n}(t_n,\mathbf 0)))))+o(t),
    \end{split}
  \end{equation}
 where $\displaystyle\frac{\del f_{i_1}}{\del t}$ stands for the partial derivative 
  of $f_{i_1}$ with respect to the first coordinate.
  In the same way, expanding the first term   $f_{i_2}(t_2,(\hdots,(f_{i_n}(t_n,\mathbf 0))))$ of \cref{expand1} in terms 
  of $t_2$, we get
   \begin{equation}
  \begin{split}
   \label{eq:alpha_est1}
    \alpha(t) &= 
    f_{i_3}(t_3,(\hdots,(f_{i_n}(t_n,\mathbf 0))))
    + t_2 \frac{\del f_{i_2}}{\del t}(0, f_{i_3}(t_3,(\hdots,(f_{i_n}(t_n,\mathbf 0))))) +o(t) \\
    &\quad + t_1 \frac{\del f_{i_1}}{\del t}(0,
    f_{i_2}(t_2,(\hdots,(f_{i_n}(t_n,\mathbf 0)))))+o(t).
    \end{split}
  \end{equation}
Continuing the expansion in this manner, we 
  obtain
    \begin{equation} 
  \begin{split}
  \label{eq:alpha_est2}
    \alpha(t) = &f_{i_n}(t_n,\mathbf 0)
    +\sum_{j=1}^{n-1} \left( t_j \frac{\del f_{i_j}}{\del t}
    (0,f_{i_{j+1}}(t_{j+1},(\hdots,(f_{i_n}(t_n,\mathbf 0))))\right)\\
    = &t_n \frac{\del f_{i_n}}{\del t}(0,\mathbf 0)
    +\sum_{j=1}^{n-1} t_j \frac{\del f_{i_j}}{\del t} (0,f_{i_{j+1}}(t_{j+1},(\hdots,(f_{i_n}(t_n,\mathbf 0))))) +o(t).
  \end{split}
  \end{equation}
Since $\displaystyle v_{i_j}(t) = \frac{\del}{\del t} f_{i_j} (t,\mathbf 0)$, 
  we have $f_{i_j}(t_i,\mathbf 0) = t_i v_{i_j}(0)+o(t)$.
  Then we have 
   \begin{equation}
   \label{last alpha}
  \begin{aligned}
    f_{i_{n-1}}&(t_{n-1},f_{i_n}(t_n,\mathbf 0))
    =  f_{i_{n-1}}(t_{n-1}, t_n
    v_{i_n}(0)+o(t))\\
   = & f_{i_{n-1}}(t_{n-1},\mathbf 0) 
   +d_{\mathbf x} f_{i_{n-1}}(t_{n-1},\mathbf 0) (t_nv_{i_n}(0)+o(t))+o(t)\\
   =&t_{n-1}v_{i_{n-1}}(0) +t_nd_{\mathbf x} f_{i_{n-1}}(t_{n-1},\mathbf 0)(v_{i_n}(0))+o(t)
    \end{aligned}
  \end{equation}
 where $(d_{\mathbf x}f_j)(t,\mathbf x)$ denotes the differential of $f_j$ with respect to $\mathbf x$ at $(t,\mathbf x)$.
   Since $f_{i_n}$ was assumed to be piecewise $C^1$ with respect to $\mathbf x$ at $\mathbf 0$,  the matrix
  $d_{\mathbf x} f_{i_{n-1}}(t_{n-1},\mathbf 0)$ is differentiable with respect to $t_{n-1}$.
  Therefore we have 
  \begin{equation}
  \label{diff}
  \begin{aligned}
  \frac{\partial f_{i_{n-1}}}{\partial t}(0, f_{i_n}(t_n,\mathbf 0))=v_{i_{n-1}}(0)+t_n(\frac{\partial}{\partial t}d_{\mathbf x} f_{i_{n-1}}(t_{n-1},\mathbf 0)) +o(1)
  \end{aligned}
  \end{equation}
   We can continue this computation 
  to replace the composition terms in
  \cref{eq:alpha_est2} with the following form:
%
 \begin{equation}
 \label{eq:alpha_est3}
 \begin{aligned}
\alpha'(0)&= \lim_{t\to 0}\frac{\alpha(t)}{t}\\
&=\lim_{t\to0}\sum_{j=1}^n \frac{t_n}{t} v_{i_j}(0).
\end{aligned}
 \end{equation}
 Although the choice of $i_1, \dots , i_n$ may vary with $t$, the assumption of  compactness  in the statement implies that $o(t)$ and $o(1)$ which appear in \cref{last alpha}  and \cref{diff} give a single $o(1)$ term in the computation of  $\displaystyle \lim_{t\to 0}\frac{\alpha(t)}{t}$ above.
From the same compactness, it follows that  the limit lies in $\Span(\{v_i(0)\})$.
%
Thus we have completed the proof.
%
\end{proof}

\subsection{Proof of \cref{BN}}
Now we return to the situation of \cref{BN}.
Let  $\alpha(t)$ be a harmonic stretch ray along $\lambda$ tangent to $v$ in $T_x \teich(S)$.
For each maximal chain-recurrent geodesic lamination $\mu$ containing $\lambda^\CR$, there is a stretch ray $\str_x(\mu, t)$ starting at $x$.
We shall apply \cref{flows} by setting $\{f_i\}$ to be $\{\str_y(\mu,t)\mid y \in U, \mu \in \mcr(\lambda^\CR)\}$ for an open neighbourhood $U$ of $x$.
For $\mu \in \mcr(\lambda^\CR)$, we define $v_\mu(t)$ to be the tangent vector $\frac{d}{dt}\str_x(\mu,t)$.
Since $\str_x(\mu,t)$ is a geodesic ray, we see that $\Vert v_\mu\Vert_\th=1$ and $v_\mu(0)=\v_x(\mu)$.

Let $\alpha(t)$ be a harmonic stretch ray along  a chain-recurrent geodesic lamination $\lambda$ starting at $x \in \teich(S)$ as in \cref{PW}.
By \cref{concatenation}, there is $n \in \naturals$ such that $x$ and $\alpha(t)$ can be joined by a concatenation of (at most) $n$ stretch paths along maximal chain-recurrent geodesic laminations containing $\lambda$.
Therefore, to show that $\{\str_{(\cdot)}(\mu, \cdot) \mid \mu \in \mcr(\lambda^\CR)\}$ satisfies the hypotheses of \cref{flows}, it remains to show that for any fixed maximal chain-recurrent geodesic lamination $\mu$, the map $\str_{x}(\mu, r)$ is $C^1$ with respect to $t$ and piecewise $C^1$ with respect to $x$, and that the families $\{\str_{x}(\mu, t) \mid \mu \in \mcr(\lambda^\CR); x \in K, t \in J\}$ and $ \{\frac{d}{dt}(\str_x(\mu,t))\mid \mu \in \mcr(\lambda^\CR); x \in K, t \in J\}$ are compact in the uniform topology.

\begin{lemma}
The stretch map $\str_{x}(\mu, t)$ is  $C^1$ with respect to $t$ and piecewise $C^1$ with respect to $x$ near any point $x \in \teich(S)$.
\end{lemma}
\begin{proof}
The stretch ray $\str_x(\mu,t)$ is shown to be $C^1$ with respect to $t$ in \cite{ThM}.
In fact, it is analytic, as was shown in \cite{PW}.

Fix a maximal chain-recurrent geodesic lamination $\mu$.
Then each point $x\in \teich(S)$ defines a horocyclic foliation $F_x$ transverse to $\mu$.
The space of measured foliations is decomposed into finitely many closed conical regions (called linear fragments in \cite{ThM}), and the transition function is piecewise linear.
The decomposition can be taken independently of  the point $x$ in $\teich(S)$.
Restricted to each region, the correspondence between horocyclic foliations and points in Teichm\"{u}ller space is $C^1$ and its differential is $C^1$ with respect to $t$.
(These properties follow from the differentiability of the cataclysm coordinates and the expressions of stretch vectors in terms of the differential of cataclysm coordinates. See \cite[\S9]{ThM}.)
This implies that $\str_x(\mu,t)$ is piecewise $C^1$ with respect to $x$.
\end{proof}
%
%
Now we turn to compactness.
\begin{lemma}
\label{closed}
For any two compact subsets $K$ of $\teich(S)$ and $J \subset [0,\infty)$, the family of geodesics  $\mathcal{SR}:=\{\str_{x}(\mu, t) \mid \mu \in \mcr(\lambda^\CR); x \in K, t \in J\}$ and the family of associated tangent vectors $\{\frac{d}{d t}(\str_x(\mu,t)) \mid \mu \in \mcr(\lambda^\CR); x \in K, t \in J\}$ are compact with respect to the uniform topologies on the sets of maps from $K \times J$ to $\teich(S)$ and $T(\teich(S))$ respectively. 
\end{lemma}
\begin{proof}
Since the space of $C^1$-rays in $\teich(S)$ and the space of their tangent vectors are metrisable, it suffices to prove  sequential compactness.
Let $(\mu_i)$ be a sequence of maximal chain-recurrent geodesic laminations in $\mcr(\lambda^\CR)$, and consider a sequence of stretch rays $\big(\str_{x_k}(\mu_k, t)\big)_{k \in \naturals}$ with $x_k \in K$. 
By the Arzel\`{a}-Ascoli theorem, passing to a subsequence, the sequence $\big(\str_{x_k}(\mu_k, t)\big)$ converges to a geodesic ray $r_x(t)$, starting at $x \in K$ to which $(x_k)$ converges, uniformly on any compact set of $[0, \infty)$.

By \cref{Hausdorff closed}, passing to a subsequence, $(\mu_i)$ converges to a geodesic lamination $\mu$ in $\mcr(\lambda^\CR)$.
Recall that $r_x(t)$ is the limit of $\str_{x_k}(\mu_k, t)$ as $k\to \infty$.
Since $\mu(x_k, \str_{x_k}(\mu_k,t))$ is  $\mu_k$, by \cref{continuity}, $\mu(x, r_x(t))$ contains $\mu$.
On the other hand, since $\mu$ is maximal, we have $\mu(x,r_x(t))=\mu$.
The maximality of $\mu$ also implies that $r_x(t)=\str_x(\mu,t)$.
This means that $r_x$ is a stretch ray along $\mu$.
Thus we have shown that the set $\mathcal{SR}$ is compact with respect to the uniform topology.

What remains to show is that the set $ \{\frac{d}{d t}(\str_x(\mu,t)) \mid \mu \in \mcr(\lambda^\CR); x \in K, t \in J\}$ is also compact with respect to the uniform topology.
Consider a sequence $\big(\frac{d}{d t}(\str_{x_k}(\mu_k,t_k))\big)$ in $\{\frac{d}{d t}(\str_x(\mu,t)) \mid \mu \in \mcr(\lambda^\CR); x \in K, t \in J\}$.
For $t_k \in [0,\infty)$, the tangent vector to $\str_{x_k}(\mu_k, t_k)$ is the stretch vector $\v_{\mu_k}(\str_{x_k}(\mu_k, t_k))$.
Since $(t_k)$ lies in $J$, it converges to $t_\infty \in J$ up to passing to a subsequence.
Then, as was shown in the preceding paragraph, passing to a subsequence, the stretch rays $\str_{x_k}(\mu_k, t)$ converge to a stretch ray $\str_{x_\infty}(\mu_\infty, t)$, where $x_\infty$ is the limit of $(x_i)$ in $\teich(S)$, and $\mu_\infty$ is the Hausdorff limit of $(mu_k)$.

We need to show that the tangent vectors $\frac{d}{d t}(\str_{x_k}(\mu_k,t_k))$ converge to $\frac{d}{dt}\str_{x_\infty}(\mu_\infty, t_\infty)$.
By translating parameters $t$, we can assume that $t_k=t_\infty=0$.
We consider the cataclysm map $\cat \colon U(\subset \mcr \times \MF(S)) \to \teich(S)$, and its differential with respect to the second factor, $d\cat \colon \mcr \times T(\mf(S)) \to T(\teich(S))$, and regard $\frac{d}{d t}(\str_{x_k}(\mu_k,t_k))$ and $\frac{d}{dt}\str_{x_\infty}(\mu_\infty, t_\infty)$ as contained in $\mcr \times T(\MF(S))$.
Then our claim is equivalent to the condition that the sequence $\big((\mu_k, F_{x_k}(\mu_k))\big)$ converges to $(\mu_\infty, F_{x_\infty}(\mu_\infty))$, which holds by the continuity of horocyclic foliations with respect to both $\teich(S)$ and $\mcr$.

%

Thus we have shown the desired sequential compactness.
\end{proof}

\begin{proof}[Proof of \cref{BN}]
Let $v$ be a unit tangent vector at $x \in \teich(S)$.
By assumption,  $v$ is a tangent vector of a harmonic stretch ray along some geodesic lamination $\lambda$.
Since such a harmonic stretch map maximally stretches only geodesic laminations contained in $\lambda$, we have $\mu(x, \alpha(t))=\lambda^\CR$.
Therefore, $\alpha(t)$ can be expressed as a finite concatenation of stretch paths along geodesic laminations contained in $\mcr(\lambda^\CR)$, as explained in  \cref{Thurston's stretch}.
Since we have already proved that the set $\mathcal{SR}$ can be taken to be $\{f_i\}$ in \cref{flows}, we see that $v=\dot{\alpha}(0)$ can be expressed as a linear combination of stretch vectors $\alpha_1 \v_x(\mu_1)+ \dots + \alpha_n\v_x(\mu_n)$ along maximal chain-recurrent geodesic laminations $\mu_1, \dots , \mu_n$ containing $\lambda^\CR$.
\end{proof}
%
%

\section{Convex structures}
\subsection{Statements of results}
Let $(\teich(S), \dth)$ be the Teichm\"{u}ller space of $S$ equipped with Thurston's metric. For each point $x \in \teich(S)$, we consider the tangent space $T_x \teich(S)$ at $x$ and its unit sphere $\mathcal S_x$ with respect to the norm $\Vert \cdot \Vert_{\mathrm{Th}}$.
We have the following characterisation of the faces of $\mathcal S_x$;

\begin{theorem}
\label{face}
For any face $F$ in $\mathcal S_x$, there is a unique chain-recurrent geodesic lamination $\lambda$ such that $F$ is expressed as 
\begin{equation*}
\begin{split}
&F=F_\lambda:=\{\alpha_1 \v_x(\lambda_1)+ \dots +\alpha_k \v_x(\lambda_k) \mid \lambda_1, \dots , \lambda_k \text{ are maximal chain-recurrent } \\
&\text{geodesic laminations containing } \lambda; \  \alpha_1, \dots , \alpha_k > 0, \alpha_1+ \dots + \alpha_k=1\}.
\end{split}
\end{equation*}
Conversely, for any chain-recurrent geodesic lamination $\lambda$, the set $F_\lambda$ defined above constitutes a face of $\mathcal S_x$.
Two faces $F_\lambda$ and $F_\mu$ associated with chain-recurrent geodesic laminations $\lambda$ and $\mu$ intersect if and only if $\lambda$ and $\mu$ do not intersect transversely.
In this case, we have $F_\lambda \cap F_\mu=F_{\lambda \cup \mu}$.
\end{theorem}

\begin{remark}
\label{characterisation}
By \cref{linear combination}, every vector in $F_\lambda$ infinitesimally most stretches $\lambda$.
Conversely, by \cref{BN}, every vector that infinitesimally most stretches $\lambda$ must be contained in $F_\lambda$.
For any chain-recurrent geodesic lamination $\mu$ properly containing $\lambda$, we can find a maximal chain-recurrent geodesic lamination which does not contain $\mu$.
Therefore $\lambda$ is the largest chain-recurrent geodesic lamination that is infinitesimally stretched by {\em all} vectors contained in $F_\lambda$.
\end{remark}


\begin{theorem}
\label{exposed}
A face $F$ is exposed if and only if the lamination $\lambda$ given in \cref{face} is an unmeasured lamination.
\end{theorem}

%
%
%

\begin{theorem}
\label{extreme}
A point $v\in \mathcal S_x$ is an extreme point if and only if $v$ is a stretch vector along a maximal chain-recurrent geodesic lamination.
\end{theorem}

We introduce the following symbol for brevity of the exposition in the  proofs of \cref{face,exposed,extreme}:
we write $\lambda \sqcap \mu$ to denote the geodesic lamination consisting of all leaves shared by $\lambda$ and $\mu$.
%
\subsection{Proof of \cref{face}}
%
%
%
For a chain-recurrent geodesic lamination $\lambda$, we define
\begin{equation*}\begin{split}
F_\lambda=&
\{ \alpha_1 \v_x(\lambda_1)+ \dots +\alpha_k \v_x(\lambda_k) \mid \lambda_1, \dots , \lambda_k \text{ are maximal chain-recurrent}\\ &\text{geodesic laminations containing } \lambda; \ 
 \alpha_1, \dots , \alpha_k >0, \alpha_1+ \dots + \alpha_k=1\},\end{split}
\end{equation*}
and we denote by $\reals_+ F_\lambda$ its cone, that is, $\reals_+F_\lambda=\{r v \mid v \in F_\lambda, \, r \in [0,\infty)\}$.
By \cref{cor:PW,BN},  $T_x \teich(S)$ is the union of $\reals_+F_\lambda$, where $\lambda$ ranges over all chain-recurrent geodesic laminations.

The proof is divided into six steps.

(1) In the first step, we shall show that $F_\lambda$ lies on $\mathcal S_x$.
For $v \in F_\lambda$, we can write $v=\alpha_1 \v_x(\lambda_1)+ \dots + \alpha_k \v_x(\lambda_k)$.
Since $\v_x(\lambda_1), \dots , \v_x(\lambda_k)$ lie on $\mathcal S_x$ and since $\alpha_1+\dots + \alpha_k=1$, by the triangle inequality, we have $\Vert v\Vert_\th\leq 1$.
Let $\lambda_0$ be a minimal component of $\lambda$.
Then there exists a positive transverse measure $m$ supported on  $\lambda_0$.
Since all of $\v_x(\lambda_1), \dots , \v_x(\lambda_k)$ infinitesimally most stretch $\lambda_0$, by \cref{linear combination}, we see that so does their linear combination, and that $d\log\ell_{(\lambda_0, m)}(\alpha\v_x(\lambda_j)+\dots + \alpha\v_x(\lambda_k))=1$.
Therefore, we have $\Vert v \Vert_\th=\sup_{\mu \in \ml(S)\setminus \{0\}}d\log\ell_\mu(x) \geq 1$.
Thus we have shown that $\Vert v \Vert_\th=1$, and since $v$ is an arbitrary point on $F_\lambda$, we have $F_\lambda \subset \mathcal S_x$.

(2) In the second step, we show that $F_\lambda$ is a face.
Let $v=\alpha_1 \v_x(\lambda_1)+ \dots +\alpha_k \v_x(\lambda_k)$ be a vector in $F_\lambda$, and suppose that it is expressed as $v=t u+(1-t)w$ for $u, w \in \mathcal S_x$ and $t \in (0,1)$.
By \cref{cor:PW,BN}, $u$ and $w$ can be expressed as $u=\beta_1 \v_x(\mu_1) + \dots + \beta_m \v_x(\mu_m)$ and $w=\gamma_1 \v_x(\nu_1)+ \dots + \gamma_n \v_x(\nu_n)$, where both $\mu_1 \sqcap \dots \sqcap\mu_m$ and $\nu_1\sqcap \dots \sqcap \nu_n$ contain chain-recurrent geodesic laminations with $\beta_1, \dots, \beta_m >0, \beta_1+ \dots +\beta_m=1$, $\gamma_1, \dots \gamma_n >0$, and $\gamma_1+ \dots +\gamma_n=1$.
We shall show that $\lambda$ is contained in both $\mu_1 \sqcap \dots \sqcap\mu_m$ and $\nu_1\sqcap \dots \sqcap \nu_n$.


Suppose that $\lambda$ is not contained in either $\mu_1 \sqcap \dots \sqcap \mu_m$ or $\nu_1 \sqcap \dots \sqcap \nu_n$; say $\mu_1 \sqcap \dots \sqcap\mu_m$.
Then there exists a $j$ such that $\mu_j$ does not contain $\lambda$.
By \cref{unique stretch}, it follows that $\lambda$ is not infinitesimally most  stretched by $\v_x(\mu_j)$.
Then, as was shown in the proof of \cref{unique stretch}, there exist a harmonic stretch ray $m_t:=r_{\mu_j}(t)$ tangent to $\v_x(\mu_j)$ at $t=0$, and a harmonic stretch map $f_t \colon (S, x) \to (S, m_t)$ such that $f_t|\lambda$ is stretched at a speed strictly less than $1$.
Since $\mu_j$ is stretched at most at speed $1$ by harmonic stretch maps for all vectors $\v_x(\mu_1), \dots, \v_x(\mu_m)$ other than $\v_x(\mu_j)$ and $\v_x(\nu_1), \dots , \v_x(\nu_n)$, we see that $\lambda$ cannot be  infinitesimally most stretched by $t(\beta_1 \v_x(\mu_1) + \dots + \beta_m \v_x(\mu_m))+(1-t)(\gamma_1 \v_x(\nu_1)+ \dots + \gamma_n \v_x(\nu_n))$.
Since  $\lambda$ is infinitesimally most stretched by $v=tu+(1-t)w$  (\cref{linear combination}), this is a contradiction. 
Thus we have shown that both $u$ and $w$ lie in $F_\lambda$. This proves that $F_\lambda$ is a face.

(3) Next we show the uniqueness of the chain-recurrent geodesic lamination $\lambda$.
Suppose that  $F_\lambda=F_\mu$ for another chain-recurrent geodesic lamination $\mu$.
Then, by \cref{characterisation}, both $\lambda$ and $\mu$ are equal to the largest chain-recurrent geodesic lamination that is infinitesimally most stretched by all vectors in $F_\lambda=F_\mu$, which means that $\lambda=\mu$.

(4) We prove the last sentence of our statement.
Suppose that $F_\lambda$ and $F_\mu$ have non-empty intersection.
Then any vector  $v$ in $F_\lambda\cap F_\mu$ infinitesimally most  stretches both $\lambda$ and $\mu$.
By \cref{unique stretch}, this implies that $\lambda$ and $\mu$ cannot intersect transversely, hence $\lambda \cup \mu$ is also infinitesimally most  stretched by $v$.
Therefore $F_\lambda \cap F_\mu \subset F_{\lambda \cup \mu}$.
Conversely, by definition, every vector  $F_{\lambda \cup \mu}$ is contained in both $F_\mu$ and $F_\lambda$.
Thus we have $F_\lambda \cap F_\mu=F_{\lambda \cup \mu}$.

(5) We show that  if $F_\lambda \cup F_\mu$ is contained in a face, then $\lambda \sqcap \mu \neq \emptyset$ and both $F_\lambda$ and $F_\mu$ are contained in $F_{(\lambda \sqcap \mu)^\CR}$.
Suppose that $F_\lambda\cup F_\mu$ is contained in a face $F$.
Take two maximal chain-recurrent geodesic laminations $\lambda_0$ containing $\lambda$ and $\mu_0$ containing $\mu$ in such a way that neither $\lambda_0 \setminus \lambda$ nor $\mu_0 \setminus \mu$ contains minimal components.
Since both $\v_x(\lambda_0) \in F_\lambda$ and $\v_x(\mu_0) \in F_\mu$ lie in $F$, so does their midpoint $v_m=(\v_x(\lambda_0)+\v_x(\mu_0))/2$.

Let $\nu$ be the largest  infinitesimally most stretched lamination for $v_m$.
By the same argument as in (2), we see that $\nu$ must be infinitesimally most stretched by both $\v_x(\lambda_0)$ and $\v_x(\mu_0)$.
By \cref{unique stretch}, $\nu$ can intersect transversely neither $\lambda_0$ nor $\mu_0$.
Since both $\lambda_0$ and $\mu_0$ are maximal, $\nu$ is contained in $\lambda_0 \sqcap \mu_0$.
Since neither $\lambda_0 \setminus \lambda$ nor $\mu_0 \setminus \mu$ contains a minimal component, the minimal components  of $\nu$ are all contained in $\lambda \sqcap \mu$.
Thus, we have shown that $\lambda \sqcap \mu \neq \emptyset$.
Since $\lambda \sqcap \mu$ is non-empty, so is $(\lambda \sqcap \mu)^\CR$, and it is straightforward that $F_\lambda \subset F_{(\lambda \sqcap \mu)^\CR}$ and $F_\mu \subset F_{(\lambda \sqcap \mu)^\CR}$ .

(6) Finally, we shall  show that every face $F$ is expressed as $F_\lambda$ for some chain-recurrent geodesic  lamination $\lambda$.
Since $\mathcal S_x=\cup_{\lambda} F_\lambda$ and by (5), we can find some $F_\lambda$ containing $F$.
Since the intersection of two faces $F_\lambda$ and $F_\mu$ is either empty or equal to $F_{\mu\cup \lambda}$, there is a minimal $F_\lambda$ containing $F$.
We shall prove that $F=F_\lambda$.

Suppose, seeking a contradiction, that $F$ is a proper subface of $F_\lambda$.
Then $F$ lies on the boundary of $F_\lambda$, and there is a stretch vector in $F_\lambda$ not contained in $F$.
Since $F$ is contained in $F_\lambda$, every vector $v$ in $F$ can be expressed as a linear combination $\alpha_1\v_x(\lambda_1)+\dots +\alpha_p \v_x(\lambda_p)$ for some positive $\alpha_1, \dots , \alpha_p$ with $\alpha_1+\dots + \alpha_p=1$, and $\lambda_1, \dots , \lambda_p \supset \lambda$.
If $F$ is also contained in $F_\mu$ for some chain-recurrent geodesic  lamination $\mu$ which is not contained in $\lambda$, then, by (4), $F$ is contained in $F_{\lambda \cup \mu}$. But this contradicts the minimality of $F_\lambda$.
Therefore, for an interior point $v$ of $F$ and its expression as above, we have $(\lambda_1 \sqcap \dots \sqcap \lambda_p)^\CR=\lambda$.

Since $F$ lies on the boundary of $F_\lambda$ whereas it is not the intersection of $F_\lambda$ with another face $F_\mu$, the only possibility is that $F$ is contained in  a limit of $F_{\mu_i}$ for some chain-recurrent geodesic laminations $\mu_i$ not contained in $\lambda$ in such a way that the interior of $F_\lambda$ is disjoint from the limit.
Let $\mu_\infty$ be the Hausdorff limit of the $\mu_i$, up to passing to a subsequence. The limit is also chain-recurrent since it is a Hausdorff limit of chain-recurrent geodesic laminations.
Then, as was shown in the proof of \cref{closed}, $\mu_\infty$ is infinitesimally most  stretched by $v$.
By \cref{unique stretch}, $\mu_\infty$ can intersect none of $\lambda_1, \dots, \lambda_p$ transversely, hence $\mu_\infty$ is contained in $\lambda$.

Now let $\hat \lambda$ be any maximal chain-recurrent geodesic lamination containing $\lambda$.
Since the Hausdorff limit $\mu_\infty$ is contained in $\lambda \subset \hat \lambda$ and since $\hat \lambda$ is chain-recurrent, we can enlarge $\mu_i$ to $\hat \mu_i \in F_{\mu_i}$ in such a way that the Hausdorff limit of $\hat\mu_i$ is $\hat \lambda$.
Then $\v_x(\hat\lambda)$ is also a limit point of $F_{\mu_i}$.
Since every point in  $F_\lambda$ is a linear combination of such vectors, we see that an interior point of $F_\lambda$  is contained in the limit of $F_{\mu_i}$,  contradicting our assumption.
Thus we have proved that $F=F_\lambda$.

%

\subsection{Proof of \cref{exposed}}
Suppose first that $\lambda$ is an unmeasured lamination.
Let $m$ be a positive transverse measure supported on $\lambda$.
Recall that every point in $F_\lambda$ is infinitesimally most stretched by $\lambda$ and conversely every unit vector that infinitesimally most stretches $\lambda$ is contained in $F_\lambda$ (See \cref{characterisation}.)  
Therefore, a point  $v$ in $\mathcal S_x$ lies on $F_\lambda$ if and only if $d\log \ell_{(\lambda,m)}(v)=1$.
Taking one point $v \in F_\lambda$, we have an expression  $F_\lambda=\{v+w \in \mathcal S_x \mid d\log\ell_{(\lambda,m)}(w)=0\}$, hence $F_\lambda$ is the intersection of $\mathcal S_x$ and a hyperplane.

Conversely, suppose that  $F_\lambda$ is expressed as $\mathcal S_x \cap H$ for some hyperplane $H$.
By Theorem 5.1 of Thurston \cite{ThM}, there is a measured lamination $\mu$ such that $H$ coincides with $H_\mu:=\{v+w\mid d\log\ell_\mu(w)=0\}$.
Since $v$ can be any vector lying on $H$, we can choose $v$ to be contained in $F_\lambda$.
Then, for any minimal component $\lambda_0$ of $\lambda$ equipped with a positive transverse measure $m_0$, we have $d\log \ell_{(\lambda_0,m_0)}(w)=1$ for any $w \in H$ since $\lambda_0$ is always infinitesimally stretched by the unit speed on $F_\lambda$.
It follows that for the hyperplane $H_{\lambda_0}=\{v+w \mid d\log\ell_{(\lambda_0,m_0)}(w)=0\}$, the intersection $\mathcal S_x \cap H_{\lambda_0}$ contains $\mathcal S_x \cap H_\mu$, hence, by \cref{unique stretch}, $\mu$ contains $\lambda_0$.
Since this holds for every minimal component of $\lambda$,  the support of $\mu$ must contain every minimal component of $\lambda$.
Let $\hat \lambda$ be the union of minimal components of $\lambda$.
We shall prove that $\lambda=\hat \lambda$, which concludes the proof.

Suppose that $\lambda \setminus \hat \lambda\neq \emptyset$.
Then we should consider two possibilities: (a) one is when $(\lambda \setminus \hat\lambda)$ is disjoint from  $\mu$, and (b) the other one is when $(\lambda \setminus \hat \lambda)$ intersects $\mu$ transversely.
Now, there is a maximal chain-recurrent geodesic lamination $\mu'$ containing the support of $\mu$ but intersecting $\lambda$ transversely in both cases (a) and (b).
Then by definition, we have $\v_x(\mu') \in H_\mu\cap \mathcal S_x=F_\lambda$. But since $\lambda$ intersects $\mu'$ transversely, we see, by \cref{unique stretch}, that  $\v_x(\mu')$ cannot lie in $F_\lambda$.
This is a contradiction, and thus we have shown that $\lambda=\hat \lambda$.

\subsection{Proof of \cref{extreme}}
To prove the \lq only if' part, suppose that $v$ is  an extreme point.
By \cref{face}, $v$ is contained in $F_\lambda$ for some chain-recurrent geodesic  lamination $\lambda$.
Then $v$ is expressed as $v=\alpha_1 \v_x(\lambda_1)+ \dots + \alpha_n \v_x(\lambda_n)$ for some distinct maximal chain-recurrent geodesic  laminations $\lambda_1, \dots , \lambda_n$ containing $\lambda$ and some positive $\alpha_1, \dots , \alpha_n$ with $\alpha_1+\dots + \alpha_n=1$.
Suppose that $n > 1$.
Then by \cref{unique stretch}, $\lambda_1$ cannot be infinitesimally most stretched by $\v_x(\lambda_2)$.
We consider two points $u$ and $w$ in $F_\lambda$ defined by $u=(\alpha_1-\epsilon)\v_x(\lambda_1)+(\alpha_2+\epsilon)\v_x(\lambda_2)+ \alpha_3 \v_x(\lambda_3)+\dots + \alpha_n\v_x(\lambda_n)$ and $w=(\alpha_1+\epsilon)\v_x(\lambda_1)+(\alpha_2-\epsilon)\v_x(\lambda_2)+\alpha_3\v_x(\lambda_3)+\dots + \alpha_n\v_x(\lambda_n)$ for sufficiently small positive $\epsilon$.
If $u=w$, then we have $\v_x(\lambda_1)=\v_x(\lambda_2)$,  contradicting the fact that $\lambda_1$ is not infinitesimally most stretched by $\v_x(\lambda_2)$.
Since  $v=tu+(1-t)w$  for $t=1/2$, we see that $v$ is contained in the interior of a segment connecting $u$ with $w$.
Hence $v$ cannot be an extreme point, contradicting the assumption.
Therefore,  the only possibility is $n=1$, and we have proved that $v=\v_x(\lambda_1)$ is a stretch vector.

Now we turn to the \lq if' part.
Let $\lambda$ be a maximal chain-recurrent geodesic lamination.
We shall show that the stretch vector $\v_x(\lambda)$ is an extreme point.
Suppose that $\v_x(\lambda)$ is contained in the interior of a segment $[u, w]$ with distinct two points $u, w\in \mathcal S_x$.
Then $\v_x(\lambda)$ is expressed as $\v_x(\lambda)=tu+(1-t)w$ for $t \in (0,1)$.
This implies that the segment connecting $u$ and $w$ lies on $\mathcal S_x$, hence it is contained in a face.
By \cref{face}, there is a face  $F_\mu$ for a chain-recurrent geodesic lamination $\mu$ which contains both $u$ and $w$.
Therefore we can express $u$ and $w$ by $u=\beta_1 \v_x(\mu_1)+ \dots \beta_p\v_x(\mu_p)$ and $w=\gamma_1 \v_x(\nu_1)+\dots + \gamma_q\v_x(\nu_q)$ for positive $\beta_1, \dots , \beta_p; \gamma_1, \dots , \gamma_q$ with $\beta_1+\dots + \beta_p=\gamma_1+ \dots + \gamma_q=1$, and maximal chain-recurrent geodesic laminations $\mu_1, \dots , \mu_p, \nu_1, \dots , \nu_q$ all containing $\mu$.
Then we have $\v_x(\lambda)=t\beta_1 \v_x(\mu_1)+ \dots +t\beta_p\v_x(\mu_p)+(1-t)\gamma_1 \v_x(\nu_1)+\dots + (1-t)\gamma_q\v_x(\nu_q)$.
The left hand side has $\lambda$ as the largest  infinitesimally stretched chain-recurrent geodesic lamination whereas that of the right hand side is $\mu_1 \sqcap \dots \sqcap \mu_p \sqcap \nu_1\sqcap \dots \sqcap \nu_q$.
By the uniqueness of  the largest infinitesimally most stretched chain-recurrent geodesic lamination, $\lambda=\mu_1 \sqcap \dots \sqcap \mu_p \sqcap \nu_1\sqcap \dots \sqcap \nu_q$.
It follows that $p=q=1$ and $\mu_1=\nu_1$, meaning that $u=w$.
This is a contradiction.
%

\section{Infinitesimal rigidity}
As an application of \cref{face,extreme}, we shall give an alternative and shorter approach to two results in \cite{Infinitesimal}:  the topological  infinitesimal rigidity  result (Theorem  1.8)  and the equivariance of stretch vectors (Corollary 1.14). Unlike the approach in \cite{Infinitesimal}, the new approach does not need to  invoke cotangent spaces.
The two results follow from the following theorem.

\begin{theorem}
\label{rigid}
Let $f \colon T_x \teich(S) \to T_y \teich(S)$ be a linear isometry between tangent spaces of the Teichm\"{u}ller space of $S$.
Then  there is a diffeomorphism $g \colon S \to S$ such that $f(F_\lambda)=F_{g(\lambda)}$ for all chain-recurrent geodesic laminaions $\lambda$.
In particular if $\lambda$ is maximal, then we have $f(\v_\lambda)=\v_{g(\lambda)}$.
\end{theorem}

To prove this, we need to introduce the following notion (\cref{def:exposed}) and three lemmas (\cref{height,intersection,scc}).

\begin{definition}[Exposed face] \label{def:exposed}
An exposed face $F$ of $\mathcal S_x$ is said to have {\em exposed face height} $n$ if there is a descending sequence $F=F_0 \supsetneqq F_1 \supsetneqq \dots \supsetneqq F_n$ of exposed faces which is the longest among such sequences.
\end{definition}

\begin{lemma}
\label{height}
Let $\lambda$ be a measured lamination and let $F_{|\lambda|}$ be the face it defines by \cref{face}.
Then $F_{|\lambda|}$ has exposed face height $3g-4$ if and only if either $\lambda$ is  a weighted simple closed curve or the minimal  supporting surface of $\lambda$ is either a torus with one hole or a sphere with four holes.
\end{lemma}
\begin{proof}
We first observe that by \cref{exposed},  the exposed faces correspond one-to-one to the unmeasured laminations on $S$, and that from the definition, given two measured laminations $\lambda$ and $\mu$,  the inclusion $F_{|\lambda|} \subsetneqq F_{|\mu|}$  holds if and only if $|\lambda|\supsetneqq |\mu|$.
Therefore a descending sequence of length $n$ corresponds to an ascending sequence of supports of measured laminations of length $n$.
Since a measured lamination can have at most $3g-3$ components, such an ascending sequence always has length at most $3g-4$.

If $\lambda$ is a weighted simple closed curve, then we can extend $|\lambda|$ to  a pants decomposition $c_1\sqcup \dots \sqcup c_{3g-3}$ of $S$ with $c_1=|\lambda|$.
This shows that $F_{|\lambda|}$ has exposed face height at least $3g-4$, and  hence equal to $3g-4$ by a remark in the preceding paragraph.
If $\lambda$ has a torus with one hole $T$ as minimal supporting surface, then, adding to $\lambda$ a maximal system of disjoint closed geodesics in $S \setminus \mathrm{Int} T$, we get a measured lamination consisting of $3g-3$ components.
This gives an ascending sequence of $3g-3$ supports of measured laminations.
The same holds in the case when $\lambda$ has a sphere with four holes as its minimal supporting surface.
Therefore in both cases $F_{|\lambda|}$ has exposed face height equal to $3g-4$, as in the case of weighted simple closed curve.
Thus we have proved the \lq if ' part.

We show the \lq only if' part by showing the contrapositive. Suppose that $\lambda$ is  neither a weighted simple closed curve nor has a one-holed torus or a sphere with four holes as its minimal supporting surface.
Let $\Sigma$ be its minimal supporting surface.
Since a pants decomposition of $\Sigma$ contains at least two simple closed curves in the interior of $\Sigma$, a measured lamination in $S \setminus \mathrm{Int} \Sigma$ can have  components fewer than $3g-4$.
Therefore the longest possible ascending sequence starting from $\lambda$ has length less than $3g-3$. This shows that the exposed face height of $F_{|\lambda|}$ is less than $3g-4$.
\end{proof}

\begin{lemma}
\label{intersection}
For two chain-recurrent laminations $\lambda$ and $\mu$, the faces $F_\lambda$ and $F_\mu$ have non-empty intersection if and only if $\lambda$ and $\mu$ do not intersect transversely.
\end{lemma}
\begin{proof}
If $\lambda$ and $\mu$ do not intersect transversely, then the geodesic lamination $\lambda \cup \mu$ is also chain-recurrent, and from the definition we get $F_{\lambda\cup \mu} \subset F_\lambda \cap F_\mu$.
Therefore, in this case, $F_\lambda$ and $F_\mu$ have non-empty intersection.

Conversely, suppose that $F_\lambda \cap F_\mu\neq \emptyset$.
Since $F_\lambda \cap F_\mu$ is a face, by \cref{face}, there is a chain-recurrent geodesic lamination $\nu$ such that $F_\nu=F_\lambda \cap F_\mu$.
Since $F_\nu$ is contained in both $F_\lambda$ and $F_\mu$, by \cref{characterisation}, we see that $\nu$ contains both $\lambda$ and $\mu$.
This is possible only when $\lambda$ and $\mu$ do not intersect transversely.
\end{proof}

The last two lemmas imply the following.

\begin{lemma}
\label{scc}
Let $f \colon T_x \teich(S) \to T_y \teich(S)$ be a linear isometry as in the statement of \cref{rigid} and let $c$ be a  simple closed geodesic on $(S,x)$.
Then there is a simple closed curve $d$ such that $f(F_{|c|})=F_{|d|}$.
\end{lemma}
\begin{proof}
Since $f$ is a linear isometry, it takes an exposed face to an exposed face.
Therefore, for any measured lamination $\lambda$, there is a measured lamination $\lambda'$ such that $f(F_{|\lambda|})=F_{|\lambda'|}$.
In this situation, we shall denote $|\lambda'|$ by $f_*(|\lambda|)$.
We use the symbol $f_*(\lambda)$ to denote some measured lamination supported on $|\lambda'|$.
This is not uniquely determined of course, but we use this just for nullity of intersection number, which does not depend on the choice of a transverse measure.
Now, for a simple closed geodesic $c$, we need only to show that $f_*(|c|)$ is also a simple closed geodesic.

Since $f$ preserves the exposed face height, by \cref{height}, either $f_*(|c|)$ is a simple closed geodesic or its minimal supporting surface  is either a torus with one hole or a sphere with four holes.
We can exclude the latter two cases by an argument similar to the one used in \S5 of \cite{alb} as follows.
Suppose, seeking a contradiction, that $f_*(|c|)$ is not a simple closed curve.
We extend  $c$ to a pants decomposition $P$ of $S$.
Then, 
\begin{quote}

(*)  there exists a measured lamination $d$ such that $i(c,d)>0$ but $i(P\setminus c, d)=0$.

\end{quote}

As was shown in the proof of \cref{height}, $f_*(|P|)$ has $3g-3$ components.
Since we assumed that $f_*(|c|)$ is not a simple closed curve, this implies that at least one component of $f_*(|P\setminus c|)$ is a frontier component of the minimal supporting surface of $f_*(|c|)$.
This means that for any measured lamination $d$ with $i(f_*(|d|), f_*(c))>0$, we also have $i(f_*(|d|), f_*(P\setminus c))>0$.
This contradicts the fact (*), since $f_*$ preserves the nullity of intersection number by \cref{intersection}.
\end{proof}

Now we can complete the proof of \cref{rigid}.
\begin{proof}
Let $c$ be a simple closed geodesic on $(S,x)$.
By \cref{scc}, there is a simple closed geodesic $d$ on $(S,y)$ such that $F_{|d|}=f(F_{|c|})$.
By taking $c$ to $d$, we have an auto-bijection $h$ on the set of vertices $\mathcal C_0(S)$ of the curve complex $\mathcal C(S)$.
Furthermore, if $i(c,c')=0$ for two simple closed geodesics $c$ and $c'$ on $(S,x)$, then by \cref{intersection}, $i(h(c), h(c'))=0$.
This shows that $h$ can be extended to a simplicial automorphism $h \colon \mathcal C(S) \to \mathcal C(S)$.
By the famous theorem of Ivanov \cite{Iv}, there is a diffeomorphism $g \colon S \to S$ which induces $h$ on $\mathcal C(S)$.
Then we have $f(F_{|c|})=F_{|g(c)|}$ by our choice of $g$.
Evidently the same holds even when $c$ is a multi-curve.

Since any chain-recurrent geodesic lamination is a Hausdorff limit of multi-curves, then, for any chain-recurrent geodesic lamination $\lambda$, there is a sequence of multi-curves $c_i$ converging to $\lambda$ in the Hausdorff topology.
By the definition of $F_\lambda$ in \cref{face} and since  $\mathbf v_\lambda=\lim_{i \to \infty} \v_{c_i}$, we have $f(F_\lambda)=F_{g(\lambda)}$.
By \cref{face}, if $\lambda$ is maximal, then $F_\lambda$ consists of only one point $\mathbf v_\lambda$.
Therefore $f(\mathbf v_\lambda)=\v_{g(\lambda)}$ in this case.
\end{proof}

\bigskip
\noindent Assaf Bar-Natan,
Model Six Operations Inc.,
1500-701 W Georgia St,
Vancouver, BC V7Y 1C6, Canada
\\
email: assaf.barnatan@gmail.com

\medskip

\noindent Ken'ichi Ohshika,
Department of Mathematics,
Gakushuin University,
Mejiro, Toshima-ku, Tokyo, Japan, and
Max-Planck-Institut für Mathematik,
Vivatsgasse 7, 53111 Bonn, Germany
\\
 email: ohshika@math.gakushuin.ac.jp
 
 \medskip
  
 \noindent Athanase Papadopoulos,
Institut de Recherche Mathématique Avancée
(Université de Strasbourg et CNRS),
7 rue René Descartes,
67084 Strasbourg Cedex France, and
Max-Planck-Institut für Mathematik,
Vivatsgasse 7, 53111 Bonn, Germany
\\
email: papadop@math.unistra.fr

\medskip

\end{document}